\documentclass{article}
\usepackage{amsmath,url,amsthm,afterpage}
\RequirePackage[colorlinks,citecolor=blue,urlcolor=blue]{hyperref}
\theoremstyle{plain}
\newtheorem{theorem}{Theorem}
\DeclareMathOperator{\lcm}{lcm}
\newcommand{\seqnum}[1]{\href{http://oeis.org/#1}{#1}}
\allowdisplaybreaks

\title{Counting toroidal binary arrays, II}
\author{S. N. Ethier\thanks{Department of Mathematics, University of Utah, 155 South 1400 East, Salt Lake City, UT 84112 USA. \href{mailto:ethier@math.utah.edu}{ethier@math.utah.edu}.  Partially supported by a grant from the Simons Foundation (209632).} \ and Jiyeon Lee\thanks{Department of Statistics, Yeungnam University, 214-1 Daedong, Kyeongsan, Kyeongbuk 712-749, South Korea. \href{mailto:leejy@yu.ac.kr}{leejy@yu.ac.kr}.  Supported by the Basic Science Research Program through the National Research Foundation of Korea (NRF) funded by the Ministry of Science, ICT \& Future Planning (No.~2013R1A1A3A04007670).}}
\date{}
\begin{document}
\maketitle

\begin{abstract}
We derive formulas for $(i)$ the number of toroidal $n\times n$ binary arrays,
allowing rotation of rows and/or columns as well as matrix transposition, and $(ii)$ the number of toroidal $n\times n$ binary arrays, allowing rotation and/or reflection of rows and/or columns as well as matrix transposition.\medskip

\noindent 2010 \emph{Mathematics Subject Classification}: Primary 05A05.\medskip

\noindent \emph{Keywords}:
toroidal array, Euler's phi function, group action, orbit, P\'olya's enumeration theorem.
\end{abstract}

\vglue0.5in

\section{Introduction}

A previous paper \cite{E} found the number of (distinct) toroidal $m\times n$ binary arrays, allowing rotation of rows and/or columns, to be
\begin{equation}\label{a(m,n)}
a(m,n):=\frac{1}{mn}\sum_{c\,|\,m}\;\sum_{d\,|\,n}\varphi(c)\varphi(d)\,2^{mn/\lcm(c,d)},
\end{equation}
where $\varphi$ is Euler's phi function and lcm stands for least common multiple.  This is \seqnum{A184271} in the \textit{On-Line Encyclopedia of Integer Sequences} \cite{S}.  The main diagonal is \seqnum{A179043}.  It was also shown that, allowing rotation and/or  reflection of rows and/or columns, the number becomes
\begin{equation}\label{b(m,n)}
b(m,n):=b_1(m,n)+b_2(m,n)+b_3(m,n)+b_4(m,n),
\end{equation}
where
$$
b_1(m,n):=\frac{1}{4mn}\sum_{c\,|\,m}\;\sum_{d\,|\,n}\varphi(c)\varphi(d)\,2^{mn/\lcm(c,d)},
$$ 
\begin{eqnarray*}
&&\!\!\!\!\!\!\!\!\!\!\!\!b_2(m,n)\\
&:=&\frac{1}{4n}\sum_{d\,|\,n}\varphi(d)\,2^{mn/d}\\
&&\!\!\!{}+\begin{cases}(4n)^{-1}\sum'\varphi(d)(2^{(m + 1)n/(2d)} - 2^{m n/d}),&\text{if $m$ is odd;}\\
(8n)^{-1}\sum'\varphi(d)(2^{m n/(2d)}+2^{(m+2)n/(2d)}- 2\cdot 2^{m n/d}),&\text{if $m$ is even,}\end{cases}
\end{eqnarray*}
with $\sum':=\sum_{d\,|\,n:\; d\text{ is odd}}$,
$$
b_3(m,n):=b_2(n,m),
$$ 
and
$$
b_4(m,n):=\begin{cases}2^{(mn - 3)/2},&\text{if $m$ and $n$ are odd;}\\
3\cdot2^{mn/2 - 3},&\text{if $m$ and $n$ have opposite parity;}\\
7\cdot2^{m n/2 - 4},&\text{if $m$ and $n$ are even.}\end{cases}
$$
(The formula for $b_2(m,n)$ given in \cite{E} is simplified here.)
This is \seqnum{A222188} in the \textit{OEIS} \cite{S}.  The main diagonal is \seqnum{A209251}. 

Our aim here is to derive the corresponding formulas when $m=n$ and we allow matrix transposition as well.  More precisely, we show that the number of (distinct) toroidal $n\times n$ binary arrays, allowing rotation of rows and/or columns as well as matrix transposition, is
\begin{equation}\label{alpha(n)}
\alpha(n)=\frac{1}{2}\,a(n,n)+\frac{1}{2n}\,\sum_{d\,|\,n}\varphi(d)\,2^{n(n+d-2\lfloor d/2\rfloor)/(2d)},
\end{equation}
where $a(n,n)$ is from \eqref{a(m,n)}.  When we allow rotation and/or reflection of rows and/or columns as well as matrix transposition, the number becomes
\begin{eqnarray}\label{beta(n)}
\beta(n)&=&\frac{1}{2}\,b(n,n)+\frac{1}{4n}\,\sum_{d\,|\,n}\varphi(d)\,2^{n(n+d-2\lfloor d/2\rfloor)/(2d)}\nonumber\\
&&\qquad{}+\begin{cases}2^{(n^2-5)/4},&\text{if $n$ is odd;}\\5\cdot 2^{n^2/4-3},&\text{if $n$ is even,}\end{cases}
\end{eqnarray}
where $b(n,n)$ is from \eqref{b(m,n)}.  At the time of writing, sequences \eqref{alpha(n)} and \eqref{beta(n)} were not in the \textit{OEIS}.

For an alternative description, we could define a group action on the set of $n\times n$ binary arrays, which has $2^{n^2}$ elements.  If the group is generated by $\sigma$ (row rotation) and $\tau$ (column rotation), then the number of orbits is given by $a(n,n)$; see \cite{E}.  If the group is generated by $\sigma$, $\tau$, and $\zeta$ (matrix transposition), then the number of orbits is given by $\alpha(n)$;  see Theorem~\ref{alpha-thm} below.  If the group is generated by $\sigma$, $\tau$, $\rho$ (row reflection), and $\theta$ (column reflection), then the number of orbits is given by $b(n,n)$; see \cite{E}.  If the group is generated by $\sigma$, $\tau$, $\rho$, $\theta$, and $\zeta$, then the number of orbits is given by $\beta(n)$; see Theorem~\ref{beta-thm} below.

Both theorems are proved using P\'olya's enumeration theorem (actually, the simplified unweighted version;  see, e.g., van Lint and Wilson \cite[Theorem~37.1, p.~524]{vW}).  

To help clarify the distinction between the various group actions, we consider the case of $3\times3$ binary arrays as in \cite{E}.  When the group is generated by $\sigma$ and $\tau$ (allowing rotation of rows and/or columns), there are 64 orbits, which were listed in \cite{E}.  When the group is generated by $\sigma$, $\tau$, and $\zeta$ (allowing rotation of rows and/or columns as well as matrix transposition), there are 44 orbits, which are listed in Table~\ref{3x3-alpha} below.  When the group is generated by $\sigma$, $\tau$, $\rho$, and $\theta$ (allowing rotation and/or reflection of rows and/or columns), there are 36 orbits, which were listed in \cite{E}.  When the group is generated by $\sigma$, $\tau$, $\rho$, $\theta$, and $\zeta$ (allowing rotation and/or reflection of rows and/or columns as well as matrix transposition), there are 26 orbits, which are listed in Table~\ref{3x3-beta} below.

\begin{table}[htb]
\caption{\label{3x3-alpha}A list of the 44 orbits of the group action in which the group generated by $\sigma$, $\tau$, and $\zeta$ acts on the set of $3\times 3$ binary arrays.  (Rows and/or columns can be rotated and matrices can be transposed.)  Each orbit is represented by its minimal element in 9-bit binary form.  Subscripts indicate orbit size.  Bars separate different numbers of 1s.}
\begin{gather*}
\begin{pmatrix}0&0&0\\0&0&0\\0&0&0\end{pmatrix}_{\!\!\!1}\bigg|
\begin{pmatrix}0&0&0\\0&0&0\\0&0&1\end{pmatrix}_{\!\!\!9}\bigg|
\begin{pmatrix}0&0&0\\0&0&0\\0&1&1\end{pmatrix}_{\!\!\!18}\begin{pmatrix}0&0&0\\0&0&1\\0&1&0\end{pmatrix}_{\!\!\!9}\begin{pmatrix}0&0&0\\0&0&1\\1&0&0\end{pmatrix}_{\!\!\!9}\bigg|\\
\begin{pmatrix}0&0&0\\0&0&0\\1&1&1\end{pmatrix}_{\!\!\!6}\begin{pmatrix}0&0&0\\0&0&1\\0&1&1\end{pmatrix}_{\!\!\!9}\begin{pmatrix}0&0&0\\0&0&1\\1&0&1\end{pmatrix}_{\!\!\!18}\begin{pmatrix}0&0&0\\0&0&1\\1&1&0\end{pmatrix}_{\!\!\!18}\begin{pmatrix}0&0&0\\0&1&1\\0&1&0\end{pmatrix}_{\!\!\!9}\begin{pmatrix}0&0&0\\0&1&1\\1&0&0\end{pmatrix}_{\!\!\!18}\begin{pmatrix}0&0&1\\0&1&0\\1&0&0\end{pmatrix}_{\!\!\!3}\begin{pmatrix}0&0&1\\1&0&0\\0&1&0\end{pmatrix}_{\!\!\!3}\bigg|\\
\begin{pmatrix}0&0&0\\0&0&1\\1&1&1\end{pmatrix}_{\!\!\!18}\begin{pmatrix}0&0&0\\0&1&1\\0&1&1\end{pmatrix}_{\!\!\!9}\begin{pmatrix}0&0&0\\0&1&1\\1&0&1\end{pmatrix}_{\!\!\!18}\begin{pmatrix}0&0&0\\0&1&1\\1&1&0\end{pmatrix}_{\!\!\!18}\begin{pmatrix}0&0&0\\1&1&1\\0&0&1\end{pmatrix}_{\!\!\!18}\begin{pmatrix}0&0&1\\0&0&1\\1&1&0\end{pmatrix}_{\!\!\!9}\\
\begin{pmatrix}0&0&1\\0&1&0\\1&0&1\end{pmatrix}_{\!\!\!9}\begin{pmatrix}0&0&1\\0&1&0\\1&1&0\end{pmatrix}_{\!\!\!9}\begin{pmatrix}0&0&1\\1&0&0\\0&1&1\end{pmatrix}_{\!\!\!18}\bigg|\begin{pmatrix}0&0&0\\0&1&1\\1&1&1\end{pmatrix}_{\!\!\!18}\begin{pmatrix}0&0&0\\1&1&1\\0&1&1\end{pmatrix}_{\!\!\!18}\begin{pmatrix}0&0&1\\0&0&1\\1&1&1\end{pmatrix}_{\!\!\!9}\\
\begin{pmatrix}0&0&1\\0&1&0\\1&1&1\end{pmatrix}_{\!\!\!18}\begin{pmatrix}0&0&1\\0&1&1\\1&1&0\end{pmatrix}_{\!\!\!9}\begin{pmatrix}0&0&1\\1&0&0\\1&1&1\end{pmatrix}_{\!\!\!18}\begin{pmatrix}0&0&1\\1&0&1\\1&1&0\end{pmatrix}_{\!\!\!18}\begin{pmatrix}0&0&1\\1&1&0\\1&0&1\end{pmatrix}_{\!\!\!9}\begin{pmatrix}0&0&1\\1&1&0\\1&1&0\end{pmatrix}_{\!\!\!9}\bigg|\\
\begin{pmatrix}0&0&0\\1&1&1\\1&1&1\end{pmatrix}_{\!\!\!6}\begin{pmatrix}0&0&1\\0&1&1\\1&1&1\end{pmatrix}_{\!\!\!9}\begin{pmatrix}0&0&1\\1&0&1\\1&1&1\end{pmatrix}_{\!\!\!18}\begin{pmatrix}0&0&1\\1&1&0\\1&1&1\end{pmatrix}_{\!\!\!18}\begin{pmatrix}0&0&1\\1&1&1\\1&0&1\end{pmatrix}_{\!\!\!9}\begin{pmatrix}0&0&1\\1&1&1\\1&1&0\end{pmatrix}_{\!\!\!18}\begin{pmatrix}0&1&1\\1&0&1\\1&1&0\end{pmatrix}_{\!\!\!3}\begin{pmatrix}0&1&1\\1&1&0\\1&0&1\end{pmatrix}_{\!\!\!3}\bigg|\\
\begin{pmatrix}0&0&1\\1&1&1\\1&1&1\end{pmatrix}_{\!\!\!18}\begin{pmatrix}0&1&1\\1&0&1\\1&1&1\end{pmatrix}_{\!\!\!9}\begin{pmatrix}0&1&1\\1&1&0\\1&1&1\end{pmatrix}_{\!\!\!9}\bigg|
\begin{pmatrix}0&1&1\\1&1&1\\1&1&1\end{pmatrix}_{\!\!\!9}\bigg|
\begin{pmatrix}1&1&1\\1&1&1\\1&1&1\end{pmatrix}_{\!\!\!1}
\end{gather*}
\end{table}
  
\begin{table}[htb]
\caption{\label{3x3-beta}A list of the 26 orbits of the group action in which the group generated by $\sigma$, $\tau$, $\rho$, $\theta$, and $\zeta$ acts on the set of $3\times 3$ binary arrays.  (Rows and/or columns can be rotated and/or reflected and matrices can be transposed.)  Each orbit is represented by its minimal element in 9-bit binary form.  Subscripts indicate orbit size.  Bars separate different numbers of 1s.}
\begin{gather*}
\begin{pmatrix}0&0&0\\0&0&0\\0&0&0\end{pmatrix}_{\!\!\!1}\bigg|
\begin{pmatrix}0&0&0\\0&0&0\\0&0&1\end{pmatrix}_{\!\!\!9}\bigg|
\begin{pmatrix}0&0&0\\0&0&0\\0&1&1\end{pmatrix}_{\!\!\!18}\begin{pmatrix}0&0&0\\0&0&1\\0&1&0\end{pmatrix}_{\!\!\!18}\bigg|\begin{pmatrix}0&0&0\\0&0&0\\1&1&1\end{pmatrix}_{\!\!\!6}\begin{pmatrix}0&0&0\\0&0&1\\0&1&1\end{pmatrix}_{\!\!\!36}\\
\begin{pmatrix}0&0&0\\0&0&1\\1&1&0\end{pmatrix}_{\!\!\!36}\begin{pmatrix}0&0&1\\0&1&0\\1&0&0\end{pmatrix}_{\!\!\!6}\bigg|
\begin{pmatrix}0&0&0\\0&0&1\\1&1&1\end{pmatrix}_{\!\!\!36}\begin{pmatrix}0&0&0\\0&1&1\\0&1&1\end{pmatrix}_{\!\!\!9}\begin{pmatrix}0&0&0\\0&1&1\\1&0&1\end{pmatrix}_{\!\!\!36}\begin{pmatrix}0&0&1\\0&0&1\\1&1&0\end{pmatrix}_{\!\!\!9}\begin{pmatrix}0&0&1\\0&1&0\\1&0&1\end{pmatrix}_{\!\!\!36}\bigg|\\
\begin{pmatrix}0&0&0\\0&1&1\\1&1&1\end{pmatrix}_{\!\!\!36}\begin{pmatrix}0&0&1\\0&0&1\\1&1&1\end{pmatrix}_{\!\!\!9}\begin{pmatrix}0&0&1\\0&1&0\\1&1&1\end{pmatrix}_{\!\!\!36}\begin{pmatrix}0&0&1\\0&1&1\\1&1&0\end{pmatrix}_{\!\!\!36}\begin{pmatrix}0&0&1\\1&1&0\\1&1&0\end{pmatrix}_{\!\!\!9}\bigg|\begin{pmatrix}0&0&0\\1&1&1\\1&1&1\end{pmatrix}_{\!\!\!6}\begin{pmatrix}0&0&1\\0&1&1\\1&1&1\end{pmatrix}_{\!\!\!36}\\
\begin{pmatrix}0&0&1\\1&1&0\\1&1&1\end{pmatrix}_{\!\!\!36}\begin{pmatrix}0&1&1\\1&0&1\\1&1&0\end{pmatrix}_{\!\!\!6}\bigg|
\begin{pmatrix}0&0&1\\1&1&1\\1&1&1\end{pmatrix}_{\!\!\!18}\begin{pmatrix}0&1&1\\1&0&1\\1&1&1\end{pmatrix}_{\!\!\!18}\bigg|
\begin{pmatrix}0&1&1\\1&1&1\\1&1&1\end{pmatrix}_{\!\!\!9}\bigg|
\begin{pmatrix}1&1&1\\1&1&1\\1&1&1\end{pmatrix}_{\!\!\!1}
\end{gather*}
\end{table}


Table \ref{values} provides numerical values for $\alpha(n)$ and $\beta(n)$ for small $n$.

\begin{table}[htb]
\caption{\label{values}The values of $\alpha(n)$ and $\beta(n)$ for $n=1,2,\ldots,12$.\medskip}
\begin{center}
\begin{tiny}
\begin{tabular}{rrr}\hline\noalign{\smallskip}
\multicolumn{1}{c}{$n$} & \multicolumn{1}{c}{$\alpha(n)$} & \multicolumn{1}{c}{$\beta(n)$} \\
\noalign{\smallskip}\hline\noalign{\smallskip}
1 & 2 & 2 \\
2 & 6 & 6 \\
3 & 44 & 26 \\
4 & 2209 & 805 \\
5 & 674384 & 172112 \\
6 & 954623404 & 239123150 \\
7 & 5744406453840 & 1436120190288 \\
8 & 144115192471496836 & 36028817512382026 \\
9 & 14925010120653819583840 & 3731252531904348833632 \\
10 & 6338253001142965335834871200 & 1584563250300891724601560272 \\
11 & 10985355337065423791175013899922368 & 2746338834266358751489231123956672 \\
12 & 77433143050453552587418968170813573149024 & 19358285762613388352671214587818634041520 \\
\noalign{\smallskip}\hline
\end{tabular}
\end{tiny}
\end{center}
\end{table}

We take this opportunity to correct a small gap in the proof of Theorem~2 in \cite{E}.  The proof assumed implicitly that $m,n\ge3$.  The theorem is correct as stated for $m,n\ge1$, so the proof is incomplete if $m$ or $n$ is 1 or 2.  Following the proof of Theorem 2 below, we supply the missing steps.

\section{Rotation of rows and columns, and matrix transposition}

Let $X_n:=\{0,1\}^{\{0,1,\ldots,n-1\}^2}$ be the set of $n\times n$ matrices of 0s and 1s, which has $2^{n^2}$
elements.  Let $\alpha(n)$ denote the number of orbits of the group action on $X_n$ by the group of order $2n^2$ generated by $\sigma$ (row rotation), $\tau$ (column rotation), and $\zeta$ (matrix transposition).  (Exception: If $n=1$, the group is of order 1.)

Informally, $\alpha(n)$ is the number of (distinct) toroidal $n\times n$ binary arrays, allowing rotation of rows and/or columns as well as matrix transposition.

\begin{theorem}\label{alpha-thm}
With $a(n,n)$ defined using \eqref{a(m,n)}, $\alpha(n)$ is given by \eqref{alpha(n)}.
\end{theorem}

\begin{proof}
Let us assume that $n\ge2$.  By P\'olya's enumeration theorem,
\begin{equation}\label{alpha(n)-prelim}
\alpha(n)=\frac{1}{2n^2}\sum_{i=0}^{n-1}\sum_{j=0}^{n-1}(2^{A_{ij}}+2^{E_{ij}}),
\end{equation}
where $A_{ij}$ (resp., $E_{ij}$) is the number of cycles in the permutation $\sigma^i\tau^j$ (resp., $\sigma^i\tau^j\zeta$); here $\sigma$ rotates the rows (row 0 becomes row 1, row 1 becomes row 2, \dots, row $n-1$ becomes row 0), $\tau$ rotates the columns,  and $\zeta$ transposes the matrix.  We know from \cite{E} that
\begin{equation}\label{a(n,n)-prelim}
a(n,n)=\frac{1}{n^2}\sum_{i=0}^{n-1}\sum_{j=0}^{n-1}2^{A_{ij}},
\end{equation}
so it remains to find $E_{ij}$.  The permutation $\zeta$ has $n$ fixed points and $\binom{n}{2}$ transpositions, so $E_{00}=n(n+1)/2$.

Notice that $\sigma$ and $\tau$ commute, whereas $\sigma\zeta=\zeta\tau$ and $\tau\zeta=\zeta\sigma$.  Let $(i,j)\in\{0,1,\ldots,n-1\}^2-\{(0,0)\}$ be arbitrary.  Then
\begin{equation*}
(\sigma^i\tau^j\zeta)^2=(\sigma^i\tau^j\zeta)(\zeta\tau^i\sigma^j)=\sigma^{i+j}\tau^{i+j},
\end{equation*}
hence
\begin{eqnarray*}
(\sigma^i\tau^j\zeta)^{2d}&=&\sigma^{(i+j)d}\tau^{(i+j)d}=((\sigma\tau)^{i+j})^d,\\
(\sigma^i\tau^j\zeta)^{2d+1}&=&\sigma^{(i+j)d+i}\tau^{(i+j)d+j}\zeta.
\end{eqnarray*}
Clearly, $(\sigma^i\tau^j\zeta)^{2d+1}$ cannot be the identity permutation, so $\sigma^i\tau^j\zeta$ is of even order.  Using the fact that, in the cyclic group $\{a,a^2,\ldots,a^{n-1},a^n=e\}$ of order $n$, $a^k$ is of order $n/\gcd(k,n)$, we find that the permutation $\sigma^i\tau^j\zeta$ is of order $2d$, where $d:=n/\gcd(i+j,n)$.  Therefore, every cycle of this permutation must have length that divides $2d$.

We claim that all cycles have length $d$ or $2d$.  Accepting that for now, let us determine how many cycles have length $d$.  A cycle that includes entry $(k,l)$ has length $d$ if $(k,l)$ is a fixed point of $(\sigma^i\tau^j\zeta)^d$.  For this to hold we must have $d$ odd (otherwise there would be no fixed points because we have excluded the case $i=j=0$ and $(i+j)d/2=\lcm(i+j,n)/2$ is not a multiple of $n$).  Since
$$
(\sigma^i\tau^j\zeta)^d=\sigma^{(i+j)(d-1)/2+i}\tau^{(i+j)(d-1)/2+j}\zeta,
$$
we must also have
\begin{equation}\label{fixed-point-eq1}
(k,l)=([l+(i+j)(d-1)/2+j],[k+(i+j)(d-1)/2+i]),
\end{equation}
where $d:=n/\gcd(i+j,n)$ and, for simplicity, $[r]:=(r\text{ mod }n)\in\{0,1,\ldots,n-1\}$.  For each $k\in\{0,1,\ldots,n-1\}$, there is a unique $l$ (namely, $l:=[k+(i+j)(d-1)/2+i]$) such that \eqref{fixed-point-eq1} holds; indeed, 
\begin{eqnarray*}
&&\!\!\!\!\!\!\!\!\!\![l+(i+j)(d-1)/2+j]\\
&=&[[k+(i+j)(d-1)/2+i]+(i+j)(d-1)/2+j]\\
&=&[k+(i+j)(d-1)/2+i+(i+j)(d-1)/2+j]\\
&=&[k+(i+j)d]\\
&=&[k+(i+j)(n/\gcd(i+j,n))]\\
&=&[k+\lcm(i+j,n)]\\
&=&k.
\end{eqnarray*}
This shows that there are $n$ fixed points of $(\sigma^i\tau^j\zeta)^d$.  Each cycle of length $d$ of $\sigma^i\tau^j\zeta$ will account for $d$ such fixed points, hence there are $n/d$ such cycles.  All remaining cycles will have length $2d$, and so there are $n(n-1)/(2d)$ of these.  The total number of cycles is therefore $n(n+1)/(2d)$.

The other possibility is that $d$ is even and all cycles have the same length, $2d$, so there are $n^2/(2d)$ of them.  Notice that $d$ is a divisor of $n$, so the contribution to 
$$
\sum_{i=0}^{n-1}\sum_{j=0}^{n-1}2^{E_{ij}}
$$
from odd $d$ is
\begin{equation}\label{odd-case}
\sum_{d\,|\,n:\; d\text{ is odd}}n\varphi(d)2^{n(n+1)/(2d)}
\end{equation}
and from even $d$ is
\begin{equation}\label{even-case}
\sum_{d\,|\,n:\; d\text{ is even}}n\varphi(d)2^{n^2/(2d)}.
\end{equation}
The reason for the coefficient $n\varphi(d)$ is that, if $d\,|\,n$, then the number of elements of the cyclic group $\{e,\sigma\tau,(\sigma\tau)^2,\ldots,(\sigma\tau)^{n-1}\}$ that are of order $d$ is $\varphi(d)$.  And for a given $(i,j)\in\{0,1,\ldots,n-1\}^2$, there are $n$ pairs $(k,l)\in\{0,1,\ldots,n-1\}^2$ such that $[k+l]=[i+j]$.  Putting \eqref{odd-case} and \eqref{even-case} together, we obtain 
\begin{equation}\label{Eij}
\sum_{i=0}^{n-1}\sum_{j=0}^{n-1}2^{E_{ij}}=\sum_{d\,|\,n}n\varphi(d)2^{n(n+d-2\lfloor d/2\rfloor)/(2d)},
\end{equation}
which, together with \eqref{alpha(n)-prelim} and \eqref{a(n,n)-prelim}, yields \eqref{alpha(n)}.
\end{proof}
.  

It remains to prove our claim that, for $(i,j)\in\{0,1,\ldots,n-1\}^2-\{(0,0)\}$, the permutation $\sigma^i\tau^j\zeta$ cannot have any cycles whose length is a proper divisor of $d:=n/\gcd(i+j,n)$.  Let $c\,|\,d$ with $1\le c<d$.  We must show that $(\sigma^i\tau^j\zeta)^c$ has no fixed points.  We can argue as above with $c$ in place of $d$.  For $(k,l)$ to be a fixed point of $(\sigma^i\tau^j\zeta)^c$ we must have $(i+j)c$ a multiple of $n$.  But $d:=n/\gcd(i+j,n)$ is the smallest integer $c$ such that $(i+j)c$ is a multiple of $n$ because $(i+j)n/\gcd(i+j,n)=\lcm(i+j,n)$.

Finally, we excluded the case $n=1$ at the beginning of the proof, but we notice that the formula \eqref{alpha(n)} gives $\alpha(1)=2$, which is correct.

\section{Rotation and reflection of rows and columns, and matrix transposition}

Let $X_n:=\{0,1\}^{\{0,1,\ldots,n-1\}^2}$ be the set of $n\times n$ matrices of 0s and 1s, which has $2^{n^2}$
elements.  Let $\beta(n)$ denote the number of orbits of the group action on $X_n$ by the group of order $8n^2$ generated by $\sigma$ (row rotation), $\tau$ (column rotation), $\rho$ (row reflection), $\theta$ (column reflection), and $\zeta$ (matrix transposition).  (Exceptions: If $n=2$, the group is of order 8; if $n=1$, the group is of order 1.)

Informally, $\beta(n)$ is the number of (distinct) toroidal $n\times n$ binary arrays, allowing rotation and/or reflection of rows and/or columns as well as matrix transposition.

\begin{theorem}\label{beta-thm}
With $b(n,n)$ defined using \eqref{b(m,n)}, $\beta(n)$ is given by \eqref{beta(n)}.
\end{theorem}

\begin{proof}
Let us assume that $n\ge3$.  (We will treat the cases $n=1$ and $n=2$ later.)  By P\'olya's enumeration theorem,
$$
\beta(n)=\frac{1}{8n^2}\sum_{i=0}^{n-1}\sum_{j=0}^{n-1}(2^{A_{ij}}+2^{B_{ij}}+2^{C_{ij}}+2^{D_{ij}}+2^{E_{ij}}+2^{F_{ij}}+2^{G_{ij}}+2^{H_{ij}}),
$$
where $A_{ij}$ (resp., $B_{ij}$, $C_{ij}$, $D_{ij}$, $E_{ij}$, $F_{ij}$, $G_{ij}$, $H_{ij}$) is the number of cycles in the permutation $\sigma^i\tau^j$ (resp., $\sigma^i\tau^j\rho$, $\sigma^i\tau^j\theta$, $\sigma^i\tau^j\rho\theta$, $\sigma^i\tau^j\zeta$, $\sigma^i\tau^j\rho\zeta$, $\sigma^i\tau^j\theta\zeta$, $\sigma^i\tau^j\rho\theta\zeta$); here $\sigma$ rotates the rows (row 0 becomes row 1, row 1 becomes row 2, \dots, row $n-1$ becomes row 0), $\tau$ rotates the columns, $\rho$ reflects the rows (rows 0 and $n-1$ are interchanged, rows 1 and $n-2$ are interchanged, \dots, rows $\lfloor n/2\rfloor-1$ and $n-\lfloor n/2\rfloor$ are interchanged), $\theta$ reflects the columns, and $\zeta$ transposes the matrix.  The order of the group generated by $\sigma$, $\tau$, $\rho$, $\theta$, and $\zeta$ is $8n^2$, using the assumption that $n\ge3$.

We have already evaluated 
$$
a(n,n)=\frac{1}{n^2}\sum_{i=0}^{n-1}\sum_{j=0}^{n-1}2^{A_{ij}},
$$
$$
\alpha(n)=\frac{1}{2n^2}\sum_{i=0}^{n-1}\sum_{j=0}^{n-1}(2^{A_{ij}}+2^{E_{ij}}),
$$
and
$$
b(n,n)=\frac{1}{4n^2}\sum_{i=0}^{n-1}\sum_{j=0}^{n-1}(2^{A_{ij}}+2^{B_{ij}}+2^{C_{ij}}+2^{D_{ij}}),
$$
so
\begin{equation}\label{beta(n)-prelim}
\beta(n)=\frac{1}{2}\,b(n,n)+\frac{1}{4}\bigg(\alpha(n)-\frac{1}{2}a(n,n)\bigg)+\frac{1}{8n^2}\sum_{i=0}^{n-1}\sum_{j=0}^{n-1}(2^{F_{ij}}+2^{G_{ij}}+2^{H_{ij}}).\quad
\end{equation}

Let us begin with
$$
\sum_{i=0}^{n-1}\sum_{j=0}^{n-1}2^{H_{ij}}.
$$
Here we are concerned with the permutations $\sigma^i\tau^j\rho\theta\zeta$ for $(i,j)\in\{0,1,\ldots,n-1\}^2$.  We will need some multiplication rules for the permutations $\sigma$, $\tau$, $\rho$, $\theta$, and $\zeta$, specifically 
$$
\sigma\tau=\tau\sigma,\quad \sigma\theta=\theta\sigma,\quad \tau\rho=\rho\tau,\quad \rho\theta=\theta\rho,\quad \sigma\rho=\rho\sigma^{-1},\quad \tau\theta=\theta\tau^{-1},
$$
and
$$
\sigma\zeta=\zeta\tau,\quad \tau\zeta=\zeta\sigma,\quad \rho\zeta=\zeta\theta,\quad \theta\zeta=\zeta\rho.
$$
It follows that (with $\tau^{-i}:=(\tau^{-1})^i$)
$$
\sigma^i\tau^j\rho\theta\zeta=\sigma^i\tau^j\zeta\theta\rho=\zeta\tau^i\sigma^j\theta\rho=\zeta\theta\tau^{-i}\sigma^j\rho=\zeta\theta\rho\tau^{-i}\sigma^{-j},
$$
and hence
\begin{equation}\label{square-1}
(\sigma^i\tau^j\rho\theta\zeta)^2=(\sigma^i\tau^j\rho\theta\zeta)(\zeta\theta\rho\tau^{-i}\sigma^{-j})=\sigma^{i-j}\tau^{-i+j}=(\sigma\tau^{-1})^{i-j}=(\sigma^{-1}\tau)^{-i+j}.
\end{equation}

In particular, if $i\in\{0,1,\ldots,n-1\}$, then the permutation $\sigma^i\tau^i\rho\theta\zeta$ is of order 2.  Furthermore, under this permutation, the entry in position $(k,l)$ moves to position $(n-1-[l+i],n-1-[k+i])$, where, as before, $[r]:=(r\text{ mod }n)\in\{0,1,\ldots,n-1\}$.  Thus, $(k,l)$ is a fixed point if and only if 
\begin{equation}\label{fixed-point-eq2}
(k,l)=(n-1-[l+i],n-1-[k+i]).
\end{equation}
For each $k\in\{0,1,\ldots,n-1\}$ there is a unique $l\in\{0,1,\ldots,n-1\}$ (namely $l:=n-1-[k+i]$) such that \eqref{fixed-point-eq2} holds; indeed,
\begin{eqnarray*}
n-1-[l+i]&=&n-1-[n-1-[k+i]+i]=n-1-[n-1-(k+i)+i]\\
&=&n-1-[n-1-k]=n-1-(n-1-k)=k.
\end{eqnarray*}
Thus, $\sigma^i\tau^i\rho\theta\zeta$ with $i\in\{0,1,\ldots,n-1\}$ is of order 2 and has exactly $n$ fixed points, hence $\binom{n}{2}$ transpositions.  This implies that $H_{ii}=n(n+1)/2$ for such $i$.

Now we let $(i,j)\in\{0,1,\ldots,n-1\}^2$ be arbitrary but with $i\ne j$.  Let us generalize \eqref{square-1} to
\begin{eqnarray*}
(\sigma^i\tau^j\rho\theta\zeta)^{2d}&=&\sigma^{(i-j)d}\tau^{(-i+j)d}=((\sigma\tau^{-1})^{i-j})^d=((\sigma^{-1}\tau)^{-i+j})^d,\\
(\sigma^i\tau^j\rho\theta\zeta)^{2d+1}&=&\sigma^{(i-j)d+i}\tau^{(-i+j)d+j}\rho\theta\zeta.
\end{eqnarray*}
The proof proceeds much like the proof of Theorem 1.  Specifically, $\sigma^i\tau^j\rho\theta\zeta$ is of order $2d$, where $d:=n/\gcd(|i-j|,n)$.  All cycles have length $d$ or $2d$.  In fact, if $d$ is odd, there are $n/d$ cycles of length $d$ and $n(n-1)/(2d)$ cycles of length $2d$.  If $d$ is even, there are $n^2/(2d)$ cycles, all of length $2d$.  And for a given $(i,j)\in\{0,1,\ldots,n-1\}^2$, there are $n$ pairs $(k,l)\in\{0,1,\ldots,n-1\}^2$ such that $[k-l]=[|i-j|]$.  We arrive at the conclusion that
\begin{equation}\label{Hij=Eij}
\sum_{i=0}^{n-1}\sum_{j=0}^{n-1}2^{H_{ij}}=\sum_{i=0}^{n-1}\sum_{j=0}^{n-1}2^{E_{ij}}.
\end{equation}

Next we evaluate 
\begin{equation}\label{Fij=Gij}
\sum_{i=0}^{n-1}\sum_{j=0}^{n-1}2^{F_{ij}}=\sum_{i=0}^{n-1}\sum_{j=0}^{n-1}2^{G_{ij}},
\end{equation}
where the equality holds by symmetry. We consider the permutations $\sigma^i\tau^j \rho \zeta$ for $(i,j)\in\{0,1,\ldots,n-1\}^2$. From the multiplication rules, it follows that
$$\sigma^i\tau^j\rho\zeta = \zeta\theta\tau^{-i}\sigma^j$$
and hence
\begin{equation}\label{square-2}
(\sigma^i\tau^j\rho\zeta)^2=(\sigma^i\tau^j\rho\zeta)(\zeta\theta\tau^{-i}\sigma^j)=\sigma^i\tau^j\rho\theta\tau^{-i}\sigma^j=\sigma^{i-j}\tau^{i+j}\rho\theta=\theta\rho\tau^{-i-j}\sigma^{-i+j},
\end{equation}
which implies
$$
(\sigma^i\tau^j\rho\zeta)^4 =(\sigma^{i-j}\tau^{i+j}\rho\theta)(\theta\rho\tau^{-i-j}\sigma^{-i+j})=e.
$$
So the permutation $\sigma^i\tau^j \rho \zeta$ is of order 4. The entry in position $(k,l)$ moves to position $([l+j],n-1-[k+i])$ under this permutation. Thus, $(k,l)\in\{0,1,\ldots,n-1\}^2$ is a fixed point of $\sigma^i\tau^j \rho \zeta$ if and only if 
\begin{equation*}
(k,l)=([l+j],n-1-[k+i]).
\end{equation*}
There is a solution $(k,l)$ if and only if there exists $l\in\{0,1,\ldots,n-1\}$ such that, with $k:=[l+j]$, we have $n-1-[k+i]=l$ or, equivalently,
\begin{equation}\label{fixed-point-eq3}
[l+i+j]=n-1-l.
\end{equation}

When $i+j \leq n-1$, \eqref{fixed-point-eq3} is equivalent to 
$$
l+i+j=n-1-l \quad \text{or} \quad l+i+j-n=n-1-l
$$ 
or to
$$
l=(n-1-i-j)/2\quad \text{or} \quad l=(2n-1-i-j)/2.
$$
If $n$ is odd and $i+j$ is odd, then  there is one fixed point, $(k,l)=([(2n-1-i+j)/2],[(2n-1-i-j)/2])$.  If $n$ is odd and $i+j$ is even, then there is one fixed point, $(k,l)=([(n-1-i+j)/2],[(n-1-i-j)/2])$.  If $n$ is even and $i+j$ is odd, then there are two fixed points, namely
\begin{eqnarray*}
(k,l)&=&([(n-1-i+j)/2],[(n-1-i-j)/2]),\\
(k,l)&=&([(2n-1-i+j)/2],[(2n-1-i-j)/2]).
\end{eqnarray*}
Finally, if $n$ is even and $i+j$ is even, then there is no fixed point.

When $i+j\ge n$, \eqref{fixed-point-eq3} is equivalent to 
$$
l+i+j-n=n-1-l \quad \text{or} \quad l+i+j-2n=n-1-l
$$ 
or to
$$
l=(2n-1-i-j)/2 \quad \text{or} \quad l=(3n-1-i-j)/2.
$$
If $n$ is odd and $i+j$ is odd, then there is one fixed point, $(k,l)=([(2n-1-i+j)/2],[(2n-1-i-j)/2])$.  If $n$ is odd and $i+j$ is even, then there is one fixed point, $(k,l)=([(n-1-i+j)/2],[(3n-1-i-j)/2])=([(n-1-i+j)/2],[(n-1-i-j)/2])$.  If $n$ is even and $i+j$ is odd, then there are two fixed points, namely
\begin{eqnarray*}
(k,l)&=&([(2n-1-i+j)/2],[(2n-1-i-j)/2]),\\
(k,l)&=&([(n-1-i+j)/2],[(n-1-i-j)/2]).
\end{eqnarray*}
Finally, if $n$ is even and $i+j$ is even, then there is no fixed point.  Notice that the results are the same for $i+j\ge n$ as for $i+j \leq n-1$.

Using \eqref{square-2}, under the permutation $(\sigma^i\tau^j \rho \zeta)^2$, the entry in position $(k,l)$ moves to position $(n-1-[k+i-j],n-1-[l+i+j])$. Thus, $(k,l)\in\{0,1,\ldots,n-1\}^2$ is a fixed point of $(\sigma^i\tau^j \rho \zeta)^2$ if and only if 
\begin{equation*}
(k,l)=(n-1-[k+i-j],n-1-[l+i+j]).
\end{equation*}
A necessary and sufficient condition on $(k,l)$ is \eqref{fixed-point-eq3} together with $[k+i-j]=n-1-k$.
Solutions have $l$ as before.  On the other hand, $k$ must satisfy
$$
k+i-j-n=n-1-k,\quad k+i-j=n-1-k,\quad\text{or}\quad k+i-j+n=n-1-k,
$$
or equivalently,
$$
k=[(n-1-i+j)/2] \quad \mbox{or} \quad k=[(2n-1-i+j)/2].
$$

If $n$ is odd, the only fixed points of $(\sigma^i\tau^j \rho \zeta)^2$ are those already shown to be fixed points of $\sigma^i\tau^j \rho \zeta$.  If $n$ is even and $i+j$ is odd, there are two fixed points of $(\sigma^i\tau^j \rho \zeta)^2$ that are not fixed points of $\sigma^i\tau^j \rho \zeta$, namely
\begin{eqnarray*}
(k,l)&=&([(n-1-i+j)/2],[(2n-1-i-j)/2]),\\
(k,l)&=&([(2n-1-i+j)/2],[(n-1-i-j)/2]).
\end{eqnarray*}
Finally, there are no fixed points when $n$ is even and $i+j$ is even.

Consequently, if $n$ is odd, then the permutation $\sigma^i\tau^j\rho\zeta$, which is of order 4, has only one fixed point.  Therefore, it has one cycle of length 1 and $(n^2-1)/4$ cycles of length 4. Thus, 
$$
\sum_{i=0}^{n-1} \sum_{j=0}^{n-1} 2^{F_{ij}}= n^2 2^{(n^2+3)/4}. 
$$
For even $n$, if $i+j$ is odd, then the permutation $\sigma^i\tau^j\rho\zeta$ has two cycles of length 1 and one cycle of length 2, and the remaining cycles are of length 4.  If $i+j$ is even, then all cycles of the permutation $\sigma^i\tau^j \rho \zeta$ are of length 4, hence there are $n^2/4$ of them.  Thus,
$$
\sum_{i=0}^{n-1} \sum_{j=0}^{n-1} 2^{F_{ij}}= \frac{1}{2}n^2 2^{(n^2-4)/4+3}+\frac{1}{2}n^2 2^{n^2/4}=5 n^2 2^{n^2/4-1}.
$$
These results, together with \eqref{alpha(n)}, \eqref{Eij}, \eqref{beta(n)-prelim}, \eqref{Hij=Eij}, and \eqref{Fij=Gij}, yield \eqref{beta(n)}.

Finally, recall that we have assumed that $n\ge3$.  We notice that the formula \eqref{beta(n)} gives $\beta(1)=2$ and $\beta(2)=6$, which are  correct, as we can see by direct enumeration.  
\end{proof}

In the derivation of \eqref{b(m,n)} in \cite{E}, the proof requires $m,n\ge3$ because the group $D_m\times D_n$ used in the application of P\'olya's enumeration theorem ($D_m$ being the dihedral group of order $2m$), is incorrect if $m$ or $n$ is 1 or 2.  If $m=2$, row rotation and row reflection are the same, so the latter is redundant.  Thus, $D_2$ should be replaced by $C_2$, the cyclic group of order 2.  The reason  \eqref{b(m,n)} is still valid is that $b_1(2,n)=b_2(2,n)$ and $b_3(2,n)=b_4(2,n)$, as is easily verified.  If $m=1$, again row reflection is redundant, so $D_1$ should be replaced by $C_1$.  Here \eqref{b(m,n)} remains valid because $b_1(1,n)=b_2(1,n)$ and $b_3(1,n)=b_4(1,n)$.  A similar remark applies to $n=2$ and $n=1$, except that here $b_1(m,2)=b_3(m,2)$, $b_2(m,2)=b_4(m,2)$, $b_1(m,1)=b_3(m,1)$, and $b_2(m,1)=b_4(m,1)$.


\end{document}